\newif\ifCOLT
\COLTfalse

\newif\ifSIAM
\SIAMfalse

\newif\ifARXIV
\ARXIVtrue

\newif\ifCOAA
\COAAfalse
\newif\ifARXIVorSIAM
\ifSIAM
\ARXIVorSIAMtrue
\fi
\ifARXIV
\ARXIVorSIAMtrue
\fi
\newif\ifCOLTorSIAM
\ifSIAM
\COLTorSIAMtrue
\fi
\ifCOLT
\COLTorSIAMtrue
\fi

\ifCOLT
\documentclass[12pt]{colt2018} %
\fi
\ifSIAM
\documentclass[review]{siamart1116}
\fi
\ifARXIV
\documentclass{article}
\usepackage{graphicx,fancyhdr,amsmath,amssymb,subfig,url}
\usepackage{amsthm}
\fi
\ifARXIVorSIAM
\usepackage[numbers,square]{natbib}
\fi
\ifCOAA
\documentclass[smallextended]{svjour3} 
\usepackage{graphicx,fancyhdr,amsmath,amssymb,subfig,url}

\usepackage[numbers,square]{natbib}
\fi

\usepackage[algo2e]{algorithm2e}

\usepackage{physics}
\usepackage{float,color}
\usepackage{enumerate}
\usepackage{thm-restate}
\usepackage{times}
\usepackage[hidelinks]{hyperref}

\newtheorem{assumption}{Assumption}

\declaretheorem[name=Theorem]{thm}
\declaretheorem[name=Lemma]{lem}

\newcommand{\zeros}{\mathbf 0}

\newcommand{\reals}{{\mbox{\bf R}}}

\newcommand{\plusLog}{\mathop{\rm log^{+}}}

\newcommand{\Expect}{\mathop{\bf E{}}}
\newcommand{\Prob}{\mathop{\bf P{}}}
\newcommand{\Geo}{\mathop{\bf Geo{}}}

\newcommand{\argmin}{\mathop{\rm argmin}}

\newcommand\ind[1]{^{(#1)}}

\def\LipFirst{L_{1}}

\def\LipFirstHat{\hat{L}_{1}}
\def\epsHat{\hat{\epsilon}}
\def\fHat{\hat{f}}
\def\LipThird{L_{3}}
\def\pLip{L_{p}}
\def\e{e}

\def\SetLip{Q}

\def\N{N}
\def\Radius{R}
\def\Dim{d}

\def\TimeCentre{T_{C}}
\def\TimeGrad{T_{1}}
\def\TimeHess{T_{2}}
\def\TimeCube{T_{3}}
\def\TimeP{T_{p}}

  \def\xBest{x_{\text{best}}}
  
\newcommand\vol[1]{\mathbf{vol}\left(#1\right)}
\newcommand\ball[2]{\mathbf{B}_{#1}\left(#2\right)}
\def\R{\mathbb{R}}

\title{Cutting plane methods can be extended into nonconvex optimization}
\ifCOLT
 \coltauthor{\Name{Oliver Hinder} \Email{ohinder@stanford.edu}\\
 \addr Management Science and Engineering Department. Stanford.
  }
\fi

\ifARXIVorSIAM
\author{Oliver Hinder\thanks{Supported at Stanford by the PACCAR Inc Stanford Graduate Fellowship and the Dantzig-Lieberman fellowship.
This paper was accepted for presentation at Conference on Learning Theory (COLT) 2018.}}
\def\cite{\citet}
\fi

\ifCOAA
\title{Cutting plane methods can be extended into nonconvex optimization\thanks{Supported at Stanford by the PACCAR Inc Stanford Graduate Fellowship and the Dantzig-Lieberman fellowship.
This paper was accepted for presentation at Conference on Learning Theory (COLT) 2018.}
}

\author{Oliver Hinder}
\def\cite{\citet}

\institute{Oliver Hinder \at
              Department of Management Science and Engineering \\
              Stanford \\
              \email{ohinder@stanford.edu}  
}

\date{Received: date / Accepted: date}
\def\cite{\citet}

\fi

\begin{document}

\maketitle

\begin{abstract}
We show that it is possible to obtain an $O(\epsilon^{-4/3})$ expected runtime --- including computational cost --- for finding $\epsilon$-stationary points of smooth nonconvex functions using cutting plane methods. This improves on the best-known epsilon dependence achieved by cubic regularized Newton of $O(\epsilon^{-3/2})$ as proved by Nesterov and Polyak (2006)\nocite{nesterov2006cubic}. Our techniques utilize the convex until proven guilty principle proposed by Carmon, Duchi, Hinder, and Sidford (2017)\nocite{carmon2017convex}.
\end{abstract}

\ifCOLTorSIAM
\fi

\ifSIAM
\fi

\section{Introduction}\label{sec:intro}
This paper focuses on finding an $\epsilon$-stationary point $x$ of the function $f : \reals^{\Dim} \rightarrow \reals$ starting from some point $x\ind{0}$, i.e.,
$$
\| \grad f(x) \| \le \epsilon
$$
under the assumptions that $f(x\ind{0}) - \inf_{z} f(z)$ is bounded below and the function has Lipschitz first and third derivatives. It is well-known that gradient descent achieves an $\epsilon^{-2}$ runtime when the first derivatives are Lipschitz. This was improved to $\epsilon^{-3/2}$ by \cite{nesterov2006cubic} using cubic regularized Newton when the  second derivatives are Lipschitz. However, each iteration of cubic regularized Newton is more expensive --- it requires Hessian evaluations and solving a linear system. This observation inspires research developing dimension-free gradient based methods that improve on the worst-case runtime of gradient descent \citep{agarwal2016finding,carmon2018accelerated,carmon2017convex,jin2017accelerated,royer2018complexity}. The iteration counts of dimension-free methods are independent of the dimension, instead depending only on measures of function regularity, i.e., Lipschitz constants. As \cite{carmon2017lower,carmon2017lowerII} showed there are fundamental dimension-free lower bounds for this problem. These lower bounds are dependent on the choice of Lipschitz assumptions for the function and the whether the algorithm evaluates the gradient or the Hessian. 

Rather than considering the high-dimensional, low accuracy regime where dimension-free gradient methods are preferred, this paper focuses on the regime where the dimension is low but we desire high accuracy. In this case, it might be acceptable that iteration costs scale polynomially with the dimension if that enables much fewer iterations. Our main result (Theorem~\ref{thm:cutting}) is an algorithm that takes
$$
\tilde{O}( ( (\TimeGrad + \Dim^{\omega}) \Dim + \TimeHess)  \epsilon^{-4/3} )
$$
time to find an $\epsilon$-stationary point, where $\TimeP$ which refers to the cost of one evaluating the function and its first $p$ derivatives and $O(d^{\omega})$ denotes the runtime for a linear system solve. For simplicity this runtime (and all other runtimes in the introduction) exclude Lipschitz constants, log factors and dependence on the gap $f(x\ind{0}) - \inf_{z} f(z)$ where $x\ind{0}$ is the starting point of the algorithm. See Table~\ref{table-results} for a comparison of our results with known results. 

\begin{table}[H]
\begin{tabular}{ |c|p{4.0cm}|r|p{3.2cm}| } 
 \hline
Lipschitz  & method & runtime & dimension-free lower bound \citep{carmon2017lower,carmon2017lowerII}  \\ 
\hline
$\grad f$  & gradient descent & $\TimeGrad \epsilon^{-2}$ & $\TimeGrad \epsilon^{-2}$  \\ 
$\grad f, \grad^2 f$  & \cite{carmon2017convex} & $\TimeGrad \epsilon^{-7/4}$ & $\TimeGrad \epsilon^{-12/7}$ \\ 
$\grad f, \grad^3 f$    &  \cite{carmon2017convex} & $\TimeGrad  \epsilon^{-5/3}$ & $\TimeGrad \epsilon^{-8/5}$  \\ 
$\grad^2 f$  & cubic reg. \cite{nesterov2006cubic} & $(\TimeHess + \Dim^{\omega}) \epsilon^{-3/2}$ & $\TimeHess \epsilon^{-3/2}$  \\ 
$\grad^p f$  & $p$th reg. \cite{birgin2017worst}. & $(\TimeP + ?) \epsilon^{-\frac{p+1}{p}}$ &$\TimeP \epsilon^{-(p+1)/p}$ \\ 
$\grad f, \grad^3 f$  &  This paper. Thm~\ref{thm:cutting}. & $( (\TimeGrad + \Dim^{\omega}) \Dim + \TimeHess)  \epsilon^{-4/3}$ &  \\ 
$\grad f, \grad^3 f$  & This paper. Thm~\ref{thmQuarticReg}. & $(\TimeCube + \Dim^{4}) \epsilon^{-4/3}$ & $\TimeCube \epsilon^{-4/3}$  \\ 
 \hline
\end{tabular}
\caption{Comparison of the runtime of different algorithms for finding stationary points of nonconvex functions. The question mark is a placeholder for the time to solve a $p$th order regularization problem.}\label{table-results}
\end{table}

To prove our results we utilize ideas from \cite{carmon2017convex}, specifically the `convex until proven guilty principle'. This is the idea that if one runs an algorithm designed for convex optimization on a nonconvex function, either:
\begin{itemize}
\item It will succeed in quickly finding a stationary point.
\item It will fail to quickly find a stationary point. In this case a certificate of nonconvexity can be obtained. This certificate of nonconvexity can be exploited to make the algorithm run quickly.
\end{itemize}
This principle allows convex optimization algorithms to be adapted to nonconvex optimization. In \cite{carmon2017convex} the convex algorithm was accelerated gradient descent; here we study cutting plane methods. 

There is a rich literature on cutting plane methods for convex optimization both theoretical \citep{atkinson1995cutting,lee2015faster,CenterGravity,shor1977cut,vaidya1989new} and empirical \citep{bahn1995cutting,goffin1997solving}. To understand when it makes sense to use a cutting plane method, suppose we wish to solve
$$
\min_{x \in \reals^{\Dim}} f(x)
$$
where $f$ is smooth, convex and the distance to optimality is bounded. To guarantee a fast runtime under these conditions we have two options: (i) we could use accelerated gradient descent or (ii) a cutting plane method. Accelerated gradient descent has an $O( \TimeGrad / \epsilon^{1/2})$ runtime \citep{nesterov1983method}; the best known cutting plane method has an $O( \TimeGrad d \log(1/\epsilon)  + \Dim^3 \log^{O(1)}(\Dim))$ runtime \citep{lee2015faster}. Therefore, if the dimension is relative low and high accuracy is desired a cutting plane method is recommended. On the other hand, if the dimension is high and low accuracy is desired accelerated gradient descent is recommended. Qualitatively, our results have a similar flavor: our cutting plane method is better than its dimension-free gradient based counterparts \citep{agarwal2016finding,carmon2018accelerated,carmon2017convex,jin2017accelerated,royer2018complexity} when the dimension is small and high accuracy is desired.

\paragraph{Outline} Section~\ref{sec:notation} describes the notation used in this paper. Section~\ref{sec:pth-order-regularization} explains why our results improve on $p$th order regularization. Section~\ref{cutting-plane-methods} reviews cutting plane methods and explains why they cannot be directly applied to nonconvex problems. Section~\ref{sec:detecting-nonconvexity} explains how to take failures of the cutting plane algorithm and use them to obtain a certificate of nonconvexity. Section~\ref{sec:exploiting-nonconvexity} explains how to exploit this certificates of nonconvexity to reduce the function value. Section~\ref{sec:main-alg} combines the components from Sections~\ref{cutting-plane-methods}-\ref{sec:exploiting-nonconvexity} to obtain our results. Section~\ref{sec:discussion} discusses possible applications for our method.

\subsection{Notation} \label{sec:notation}
Let $\Dim$ be the dimension of the problem, $\reals$ the set of real numbers, $\| \cdot \|$ denote the euclidean norm, $\ball{R}{v} := \{ x \in \reals^{\Dim} : \| x - v \| \le R \}$, $\lambda_{\min}(\cdot)$ the minimum eigenvalue of a matrix. Unless otherwise specified $\log(\cdot)$ is base $\e$ where $\e$ is the exponential constant. Let $\plusLog( \theta ) := \max\{1, \log(\theta) \}$. The value $O(d^{\omega})$ denotes the runtime for solving a linear system or computing an SVD with $\omega \in [2,3]$ being the fast matrix multiplication constant \citep{demmel2007fast}. Given a set $S \subseteq \R^{\Dim}$, $\vol{S} := \int_{S} dx$ denotes the volume of that set. The term $\Geo(p)$ denotes the geometric distribution with success probability $p \in [0,1]$.

We say that a function $f : \reals^{\Dim} \rightarrow \reals$ has $\pLip$-Lipschitz derivatives on the convex set $\SetLip \subseteq \reals^{\Dim}$ if
\begin{flalign*} 
\abs{q^{(p)}(0) - q^{(p)}(\theta)} &\le \pLip \abs{\theta} \\
q(\theta) &:= f(x + s \theta)
\end{flalign*}
for all $x \in Q$, $s \in \ball{1}{\zeros}$, and $\theta \in \{ \theta \in \reals :  x + s \theta \in Q \}$. It is well-known that this implies by the Taylor's theorem that
$$
\abs{q(0) + \theta q^{(1)}(0) + \dots +  \frac{\theta^{p}}{p!} q^{(p)}(0)   - q(\theta)} \le \frac{\pLip}{(p+1)!} \abs{\theta}^{p+1}
$$
and
$$
\abs{q^{(1)}(0) + \theta q^{(2)}(0) + \dots +  \frac{\theta^{p-1}}{(p-1)!} q^{(p)}(0)   - q^{(1)}(\theta)} \le \frac{\pLip}{p!} \abs{\theta}^{p}.
$$
Let $\TimeP$ refer to the cost of evaluating the function and its first $p$ derivatives once. This includes the cost of adding two $p$th order tensors or multiplying them by a scalar. We assume $\TimeP = \Omega(d)$. For simplicity one can think of $\TimeP = \Theta(\Dim^{p})$ but this need not be true. If the tensors associated with the $p$th derivatives are dense then $\TimeP = \Omega(\Dim^{p})$. Conversely, if the $p$th derivatives are sparse and $p > 1$ then it is possible that $\TimeP \ll \Dim^{p}$. Furthermore, if the derivatives are difficult to evaluate then it is possible that $\TimeP \gg \Dim^{p}$. Finally, let $\TimeCentre$ refer to the time for a cutting plane centre computation (see Assumption~\ref{assume:cutting-plane}).

\subsection{Review of $p$th order regularization}\label{sec:pth-order-regularization}

Since our algorithm is closely related to $p$th order regularization \citep*{birgin2017worst} with $p=3$ we feel it is useful to further discuss this method. In particular, our goal is to explain why this method does not include computation cost in its runtime. This is contrast to our method that does include computational cost.

First let us derive $p$th order regularization. Consider a $p$th order taylor series expansion of a differentiable function $f$ at the point $\bar{x}$:
$$
f(\bar{x}) + \grad f(\bar{x})^T (x - \bar{x}) + \frac{1}{2} (x - \bar{x})^T \grad^2 f (\bar{x}) (x - \bar{x}) + \dots
$$
Adding a regularization term, we obtain
\begin{flalign}\label{pth-order-regularization-function}
\tilde{f}_{p}(\bar{x}; x) := f(\bar{x}) +  \grad f(\bar{x})^T (x - \bar{x}) + \dots + \frac{2 \pLip}{(p+1)!} \| \bar{x} - x \|^{p},
\end{flalign}
where $\pLip$ is the Lipschitz constant of the $p$th order derivatives.  The function $\tilde{f}_{p}$ is an upper bound on $f$, i.e., $\tilde{f}_{p}(\bar{x};x) \ge f(x)$. We define $p$th order regularization method as any sequence $x\ind{0}, \dots, x\ind{k}$ that satisfies
\begin{flalign}\label{pth-order-regularization-conditions}
\| \grad \tilde{f}_{p}(x\ind{k}; x\ind{k+1}) \| \le \epsilon/2, \quad \quad  \tilde{f}_{p}(x\ind{k}; x\ind{k+1}) \le f(x\ind{k}).
\end{flalign}
To meet these conditions it is sufficient to set 
$$
x\ind{k+1} \gets \argmin_{x} \tilde{f}_{p}(x\ind{k}; x).
$$
This method requires 
$$
O\left(   \Delta \pLip^{1/p} \epsilon^{-\frac{p+1}{p}}  \right)
$$ 
iterations to find stationary points \citep{birgin2017worst}, with $\Delta = f(x\ind{0}) - \inf_{z} f(z)$. For $p=1$ and $p=2$ this method corresponds to gradient descent and cubic regularization respectively. Increasing $p$ improves the $\epsilon$ dependence. However, this improvement in the $\epsilon$ dependence is only with respect to the evaluation complexity --- the number of times that we compute the $1, \dots, p$ derivatives. It excludes the cost of finding a solution to \eqref{pth-order-regularization-conditions}. Finding a point satisfying \eqref{pth-order-regularization-conditions} is trivial for gradient descent and well-known for cubic regularization \citep[Section 5]{nesterov2006cubic}. Unfortunately, for $p \ge 3$ the only available methods for solving  \eqref{pth-order-regularization-conditions} have worse $\epsilon$ dependencies. For example, cubic regularization can be used to solve \eqref{pth-order-regularization-conditions} using $O(\epsilon^{-3/2})$ steps. Therefore prior to our work, no method actually improved on the $\epsilon$ dependence of cubic regularization --- if one includes computation cost not just evaluation complexity.

\section{Cutting plane methods}\label{cutting-plane-methods}

\SetKwFunction{ExploitNC}{ExploitNC}
\SetKwFunction{CuttingPlaneMethod}{CuttingPlaneMethod}
\SetKwFunction{NonconvexityCertificate}{NonconvexityCertificate}
\SetKwFunction{Centre}{Centre}
\SetKwFunction{CuttingTrustRegion}{CuttingTrustRegion}
\SetKwFunction{GuardedCuttingPlaneAlgorithm}{GuardedCuttingPlaneAlgorithm}

Cutting plane methods encompass a variety of different algorithms which can be all written in the generic framework given by Algorithm~\ref{alg:CuttingPlaneMethod}. They work by maintaining a region $S^{(t-1)}$ that contains a minimizer. At each iteration the cutting plane picks a `centre' point $x^{(t)}$ of the region $S^{(t-1)}$. At this point a cut is generated which further reduces the volume of the region. The main difference between different cutting plane methods is how they pick the centre point. For example, centre of gravity \citep{CenterGravity} picks the point
$$
\frac{\int_{S}{x ~~ dx}}{\int_{S}{dx}} 
$$
but this is different from the volumetric \citep{vaidya1989new} or analytic centre \citep{atkinson1995cutting}. The cost of each centre computation varies by method. For example, computing the centre of gravity is prohibitively expensive. However, some methods require less expensive centre computations. With this in mind, we make Assumption~\ref{assume:cutting-plane} to ensure that our method can generically handle different centre point selections. The term $1 - \tau$ represents the minimum reduction factor in the volume of $S\ind{t}$ at each iteration. For example, for the centre of gravity $\tau = \e^{-1}$ and for the Ellipsoid Method $\tau = 1 - \e^{-d/2}$ \citep{shor1977cut}. We remark that $\tau \le1/2$ for any possible method \citep{NemirovskiYu83}.

\begin{algorithm2e}[H]
\SetAlgoLined
\KwData{$\fHat,  x\ind{0}$, $\N$, $R$}
\KwResult{$S\ind{\N}, x\ind{0}, \dots, x\ind{\N}$}

$S\ind{0} \gets \ball{R}{x\ind{0}} \cap \{ x \in \reals^{d} : \grad \fHat(x\ind{0})^T (x - x\ind{0}) \le \zeros \}$\;

\For{$t = 1, \dots , \N$}{
$x\ind{t} \gets $\Centre{$S\ind{t-1}$} \;

$S\ind{t} \gets S\ind{t-1} \cap \{ x \in \reals^{d} : \grad \fHat(x\ind{t})^T (x - x\ind{t}) \le \zeros \}$
}
\Return{$S\ind{\N},x\ind{0},\dots,x\ind{N}$}

\caption{CuttingPlaneMethod}\label{alg:CuttingPlaneMethod}
\end{algorithm2e}

An astute reader might notice that Algorithm~\ref{alg:CuttingPlaneMethod} uses $\fHat$ instead of $f$. This is to avoid confusion because our results in Section~\ref{sec:main-alg} modify the original function $f$ by adding a proximal term. This new function we call $\fHat$, is the function we call Algorithm~\ref{alg:CuttingPlaneMethod} on.

\begin{assumption}\label{assume:cutting-plane}
There exists some $\tau \in (0,1/2]$ such that for all $R \in (0,\infty)$ and positive integers $\N$, Algorithm~\ref{alg:CuttingPlaneMethod} satisfies
$$
\vol{S\ind{\N}} \le (1 - \tau)^{\N} \times \vol{\ball{R}{\zeros}}.
$$
Furthermore, the time for calling the routine $\Centre$ is $\TimeCentre$.
\end{assumption}

\def\valueN{\frac{d}{\tau} \log( R / r)}
\def\valueNTwo{\frac{d}{\tau} \log\left( \LipHat R / r \right) }

From Assumption~\ref{assume:cutting-plane} we immediately derive Lemma~\ref{lem:cutting-plane}. Lemma~\ref{lem:cutting-plane} is a standard result but we include it for exposition. We use Assumption~\ref{assume:cutting-plane} to ensure our results are generic. In Section~\ref{sec:main-alg}, we substitute explicit values for $\TimeCentre$.
  
\begin{lem}\label{lem:cutting-plane}
Let $\fHat : \reals^{d} \rightarrow \reals$ be differentiable, $\N$ be a positive integer, $x\ind{0} \in \reals^{\Dim}$, and $r, R \in (0,\infty)$.
Consider Algorithm~\ref{alg:CuttingPlaneMethod}. Suppose Assumption~\ref{assume:cutting-plane} holds. If $\N \ge \valueN$ then 
$$
\vol{S\ind{\N}} \le \frac{1}{2} \vol{\ball{r}{\zeros}}.
$$
\end{lem}
\begin{proof}
Using Assumption~\ref{assume:cutting-plane}, $\frac{\vol{\ball{r}{\zeros}}}{\vol{\ball{R}{\zeros}}} = (r/R)^{\Dim}
$, and $(1 - \tau)^{1/\tau} \le \e^{-1}$ we obtain
\begin{flalign*}
\vol{S\ind{\N}} &\le \frac{1}{2} (1 - \tau)^{\N} \vol{\ball{R}{\zeros}} = \frac{1}{2} (1 - \tau)^{\N} (R / r)^{d} \vol{\ball{r}{\zeros}}  \\
&\le \frac{1}{2} \e^{-d \log(R/r)} (R / r)^{d} \vol{\ball{r}{\zeros}} = \frac{1}{2} \vol{\ball{r}{\zeros}}.
\end{flalign*}
\end{proof}

\def\smallRad{r}
\def\xBest{x_{\text{best}}}
\def\xGD{y}

Notice that so far we have not used convexity! So why is it a non-trivial task to adapt a cutting plane method to a nonconvex function? Even though by Lemma~\ref{lem:cutting-plane} we can guarantee that $\vol{S\ind{\N}}$ is small we cannot guarantee that it contains a stationary point. To understand this failure we use Figure~\ref{fig:cutting-plane-failure}. In Figure~\ref{fig:cutting-plane-failure} a cutting plane method is applied to the function $(x_{1}^2 - 1)^2 + (x_{2}^2 - 1)^2$ with centre points $x\ind{t}$ picked arbitrarily. After three cuts the method has restricted its search to the set $S^{(2)}$ which does not contain any stationary point! In convex optimization this could not happen --- by convexity the intersection of our cutting planes will always contain the optimum.

  \begin{figure}[H]
  \includegraphics[scale=0.5]{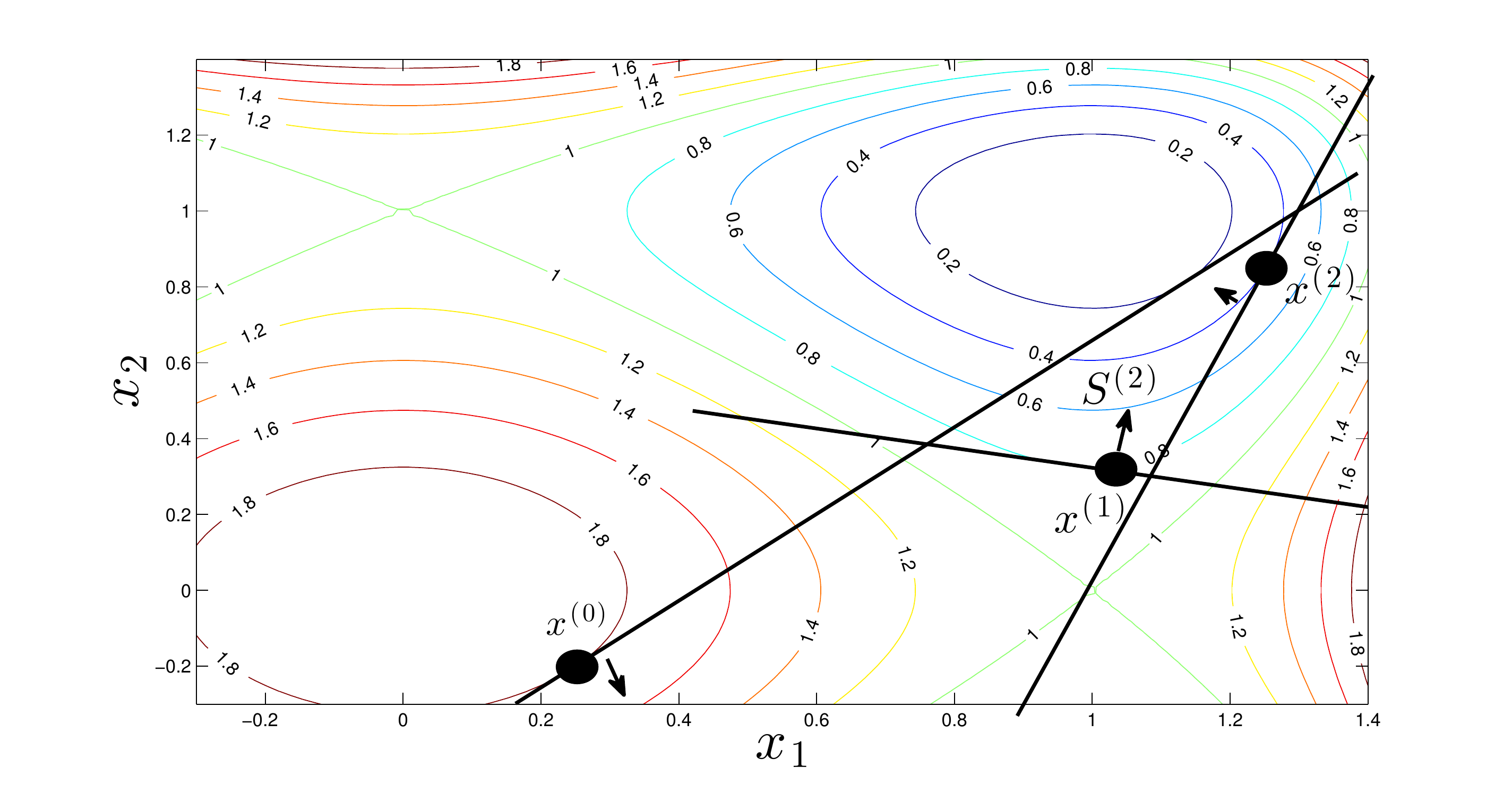}
  \caption{Failure of cutting plane methods on the function $(x_{1}^2 - 1)^2 + (x_{2}^2 - 1)^2$.}\label{fig:cutting-plane-failure}
 \end{figure}
 
 To enable the usage of cutting plane methods in nonconvex optimization, the first step is to efficiently detect these failures. This is the subject of Section~\ref{sec:detecting-nonconvexity}.

\section{Detecting nonconvexity}\label{sec:detecting-nonconvexity}

Suppose that we have run our cutting plane algorithm and we have a small set $S\ind{N}$ which we believe contains a stationary point. How can we check if it contains a stationary point? Furthermore, if it does not contain a stationary point can we produce a certificate of nonconvexity? This is the purpose of Algorithm~\ref{alg:NonconvexityCertificate}. This section is to analogous to Section~2.1 of \cite{carmon2017convex} in the sense we aim to find a certificate of nonconvexity. Our goal is to obtain a certificate of nonconvexity when a cutting plane method fails to produce a stationary point. In contrast, \cite{carmon2017convex} find a certificate when accelerated gradient descent stalls.

\begin{algorithm2e}[H]
\SetAlgoLined
\KwData{$\fHat, S\ind{N}, x\ind{0}, \dots, x\ind{\N}, \LipFirstHat, \epsHat, \Radius$}
\KwResult{$u, v, K$}
\SetKwFunction{rEq}{radius-equivalent}

$\xBest \gets \argmin_{x \in \{ x\ind{0}, \dots, x\ind{\N} \} }{ \fHat(x) } $ \;

\If{$\| \grad \fHat(\xBest) \| \le \epsHat$}{
\Return{$\xBest, \emptyset, 0$}
}

$\xGD \gets \xBest - \frac{1}{\LipFirstHat} \grad \fHat(\xBest)$

$r \gets \epsHat /  (8 \LipFirstHat)$ \\
\For{$k = 1, \dots, \infty$}{
$u \gets $ uniformly random point from $\ball{r}{y}$ \; \ifCOLT \\ \fi
\If{$u \not\in S\ind{N}$}{
$K \gets k$ \\
\textbf{break} \tcc{Lemma~\ref{lem:NonconvexityCertificateWellDefined} proves $\fHat(u) \le \fHat(\xBest)$ implying nonconvexity} 
}
}

\eIf{$\| u - x\ind{0} \| \le \Radius$}{
\tcc{Find a certificate of this nonconvexity} 
\For{$t \in \{0, \dots, T\} $}{
\If{$\fHat(u) < \fHat(x\ind{t}) + \grad \fHat(x\ind{t})^T (u - x\ind{t})$}{
\Return{$u,x\ind{t},K$} 
}
}
}{
\Return{$u,\emptyset,K$}
}

\caption{NonconvexityCertificate}\label{alg:NonconvexityCertificate}
\end{algorithm2e}

\smallskip
Algorithm~\ref{alg:NonconvexityCertificate} is combined with Algorithm~\ref{alg:CuttingPlaneMethod} in the following process:
\begin{subequations}\label{eq:call-cut-then-cert}
\begin{flalign}
& S\ind{N}, x\ind{0}, \dots, x\ind{\N} \gets \text{\CuttingPlaneMethod{$\fHat,  x\ind{0}, \N, R$}} \\
& u,v,K \gets \text{\NonconvexityCertificate{$\fHat, S\ind{N}, x\ind{0}, \dots, x\ind{\N}, \LipFirstHat, \epsHat, \Radius$}}.
\end{flalign}
\end{subequations}

Lemma~\ref{lem:NonconvexityCertificateWellDefined} summarizes possible outcomes of \eqref{eq:call-cut-then-cert}. The main idea is that if we can find a points $u$ and $v$ such that (i)
$\fHat(u) \le \fHat(v)$ and (ii) the point $u$ is not inside the halfspace $\{x \in \R^{\Dim} : \grad \fHat(v)^T (x - v) \le 0 \}$ then one obtains a certificate of nonconvexity, i.e., $\fHat(u) < \fHat(v) + \grad \fHat(v)^T (u - v)$.

\begin{figure}[H]
\includegraphics[scale=0.4,page=1]{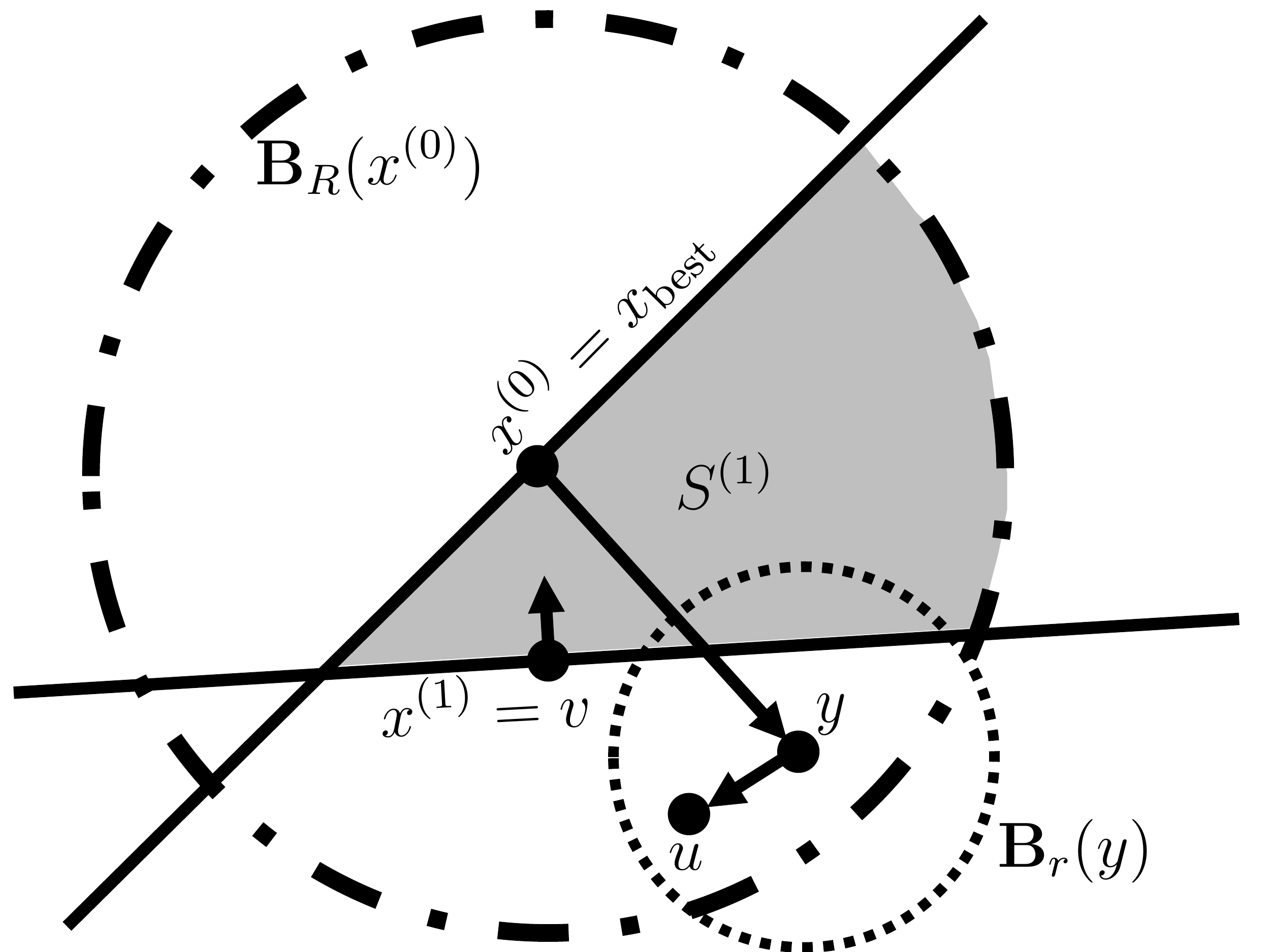}
\caption{Diagram showing an example of Algorithm~\ref{alg:NonconvexityCertificate} finding a certificate of nonconvexity. %
}\label{diagram:non-convexity}
\end{figure}

Figure~\ref{diagram:non-convexity} gives a example of Algorithm~\ref{alg:NonconvexityCertificate} detecting nonconvexity where $N = 1$, the set $S\ind{1} = \{ x \in \ball{R}{x\ind{0}} : \grad \fHat(x\ind{0})^T (x - x\ind{0}) \le 0, \grad \fHat(x\ind{1})^T (x - x\ind{1}) \}$ and $\xBest = x\ind{0}$. First the algorithm takes a gradient step from $\xBest = x\ind{0}$ to $y$. After sampling from $\ball{r}{y}$ we are at the point $u$. This $u$ is not in the set $S\ind{1}$ but $\| u - x\ind{0} \| \le R$. Therefore we must have violated some hyperplane that makes up the set $S\ind{1}$. It turns out this hyperplane corresponds to $x\ind{1}$. Therefore we set $v = x\ind{1}$ and return $u, v$ from Algorithm~\ref{alg:NonconvexityCertificate}. 

One possible issue is that after taking the gradient step we still have $y \in S\ind{\N}$. Randomly sampling from the set $\ball{r}{y}$ allows us to find a point in the nonempty set $Z := \ball{r}{y} \setminus S\ind{\N}$. If the radius $r$ is small enough then we will still have $\fHat(u) \le \fHat(\xBest)$. Each time we sample the probability the point $u$ is in the set $Z$ is equal to $\frac{\vol{Z}}{\vol{\ball{r}{y} }}$, i.e., the event $u \in Z$ is a biased coin toss. The random variable $K$ is the number of samples until the event $u \in Z$ occurs and therefore has a geometric distribution.  Due to this randomness all the results in this paper hold almost surely. We remark that this is the only place randomization is used in the paper.

\begin{lem}\label{lem:NonconvexityCertificateWellDefined}
Suppose Assumption~\ref{assume:cutting-plane} holds. Let $x\ind{0} \in \reals^{\Dim}$,  $N$ be a positive integer greater than $\frac{d}{\tau} \log\left( \frac{8 \LipFirstHat R}{\epsHat} \right)$, and $\LipFirstHat, R, \epsHat \in (0,\infty)$. Assume $\fHat : \reals^{d} \rightarrow \reals$ has $\LipFirstHat$-Lipschitz derivatives on the convex set $\SetLip \subseteq \reals^{\Dim}$. 
Let \eqref{eq:call-cut-then-cert} hold. 

Then $\fHat(u) \le \fHat(\xBest)$ where $\xBest = \argmin_{x \in \{ x\ind{0}, \dots, x\ind{\N} \} }{ \fHat(x) }$, and $K \sim \Geo(p)$ with probability of success $p \ge 1/2$. Furthermore, one of the following cases applies, 
\begin{enumerate}[(i)]
\item $v = \emptyset$, $\| \grad \fHat(u) \| \le \epsHat$
\item $v = \emptyset$, $\| u - x\ind{0} \| > \Radius$ 
\item $u$ and $v$ certify nonconvexity of $\fHat$, i.e.,
\begin{flalign}\label{non-convexity-cert}
\fHat(u) < \fHat(v) + \grad \fHat(v)^T (u - v).
\end{flalign}
\end{enumerate}
\end{lem}

\begin{proof}
First we show $\fHat(u) \le \fHat(\xBest)$.
If $u = \xBest$ this occurs trivially by definition of $\xBest$. Now,
\begin{flalign*}
\fHat(u) &\le \fHat( y )+  \grad \fHat(\xGD)^T (u - y) + \frac{\LipFirstHat \| u - y \|^2}{2} \\
 &\le \fHat( \xBest ) -  \frac{\| \grad \fHat(\xBest) \|^2}{2 \LipFirstHat}  +  \grad \fHat(\xGD)^T ( u - y )  + \frac{\LipFirstHat \|  u - y \|^2}{2} \\
&\le \fHat( \xBest ) -  \frac{\| \grad \fHat(\xBest) \|^2}{2 \LipFirstHat} + \| \grad \fHat(\xGD) \| \frac{\epsHat}{8 \LipFirstHat} + \frac{\epsHat^2}{128 \LipFirstHat}  \\
&\le \fHat( \xBest ) -  \frac{\| \grad \fHat(\xBest) \|^2}{2 \LipFirstHat} + \| \grad \fHat(\xBest) \| \frac{\epsHat}{4 \LipFirstHat} + \frac{\epsHat^2}{128 \LipFirstHat} \le \fHat( \xBest )
\end{flalign*}
where the first two transitions use the inequality $\fHat(x') \le \fHat(x) + \grad \fHat(x)^T (x' - x) + \frac{\LipFirstHat \| x' - x \|^2}{2}$, the third uses $\| u - y \| \le \smallRad \le \epsHat /  (8 \LipFirstHat)$, the fourth uses  $ \| \grad \fHat(\xGD) \| \le \| \grad \fHat(\xBest) \| +  \| \grad \fHat(\xGD) - \grad \fHat(\xBest) \| \le \| \grad \fHat(\xBest) \| + \LipFirstHat \| \xGD - \xBest \| = 2 \| \grad \fHat(\xBest) \|$, and the fifth $\| \grad \fHat(\xBest) \| \ge \epsHat$. This proves $\fHat(u) \le \fHat(\xBest)$.

Let us show $K \sim \Geo(p)$. Each event $u \not\in S\ind{N}$ occurs independently for each $k = 1, \dots, \infty$. Therefore $K \sim \Geo(p)$ with $p \ge 1/2$. 

Let us now show that one of cases (i)-(iii) holds.
If $v = \emptyset$ then clearly one of cases (i) or (ii) holds. If $v \neq \emptyset$ then $\| u - x\ind{0} \| \le \Radius$ and since $u \not\in S\ind{N}$ there exists some $t$ for which $\grad \fHat(x\ind{t})^T (u - x\ind{t}) > 0$. Since $\fHat(u) \le \fHat(x\ind{t})$ we have $\fHat(u) < \fHat(x\ind{t}) + \grad \fHat(x\ind{t})^T (u - x\ind{t})$.

Since $\vol{S} < \frac{1}{2} \vol{\ball{\smallRad}{\zeros}}$ by Lemma~\ref{lem:cutting-plane} with $r = \epsHat /  (8 \LipFirstHat)$, for each $u$ generated by sampling from $\ball{r}{y}$ the probability that $u \in S\ind{N}$ is at most $1/2$.
\end{proof}

\section{Exploiting nonconvexity}\label{sec:exploiting-nonconvexity}

Suppose that we run Algorithm~\ref{alg:NonconvexityCertificate} and find a certificate of nonconvexity. How do we use this information? This is the purpose of Algorithm~\ref{alg:ExploitNC}. In particular, we construct a function $q$ along the direction $s$ of nonconvexity of the function $f$ and then query several points on this function.
We draw on the ideas of \cite*{carmon2017convex} to provide more efficient negative curvature exploitation when the third derivatives are Lipschitz.

\def\constantOne{c}
\def\constantOneValue{-1/12 + \sqrt{97}/12}

\begin{algorithm2e}[H]
\SetAlgoLined
\KwData{$f,  c, s, \Radius$}
\KwResult{$x$}

$q(\theta) := f \left(c + \theta s \right)$\; 

$\theta_{*} \gets \argmin_{\theta \in \{12 \Radius, 9 \Radius, -9 \Radius, - 12 \Radius \}}{q(\theta)}$

\Return{$u + \theta_{*}  s$}
\caption{ExploitNC}\label{alg:ExploitNC}
\end{algorithm2e}

We need to guarantee if there is sufficient nonconvexity between $u$ and $v$ that Algorithm~\ref{alg:ExploitNC} will reduce the function value. This is the purpose of Lemma~\ref{lemNCprogress}.

\begin{restatable}{lem}{lemNCprogress}\label{lemNCprogress}
Suppose the function $q : [-12R, 12R] \rightarrow \reals$ has $\LipThird$-Lipschitz continuous third derivatives, and for some $\gamma \in [-1,1]$ we have $q''(R \gamma) \le -21 \LipThird R^2$ then
$$
\min\{ q(12 R), q(9 R), q(-9 R), q(-12 R) \} \le q(0) - 536 \LipThird R^4.
$$ 
\end{restatable}

The proof of Lemma~\ref{lemNCprogress} is given in Section~\ref{sec:lem:NC-progress}. We remark that Lemma~\ref{lemNCprogress} is similar to Lemma 5 in \cite{carmon2017convex}. The main difference is that the progress is guaranteed with respect to the function value at the origin rather than the maximum of two function values. This is critical to our result.

\begin{figure}[H]
\includegraphics[scale=0.4]{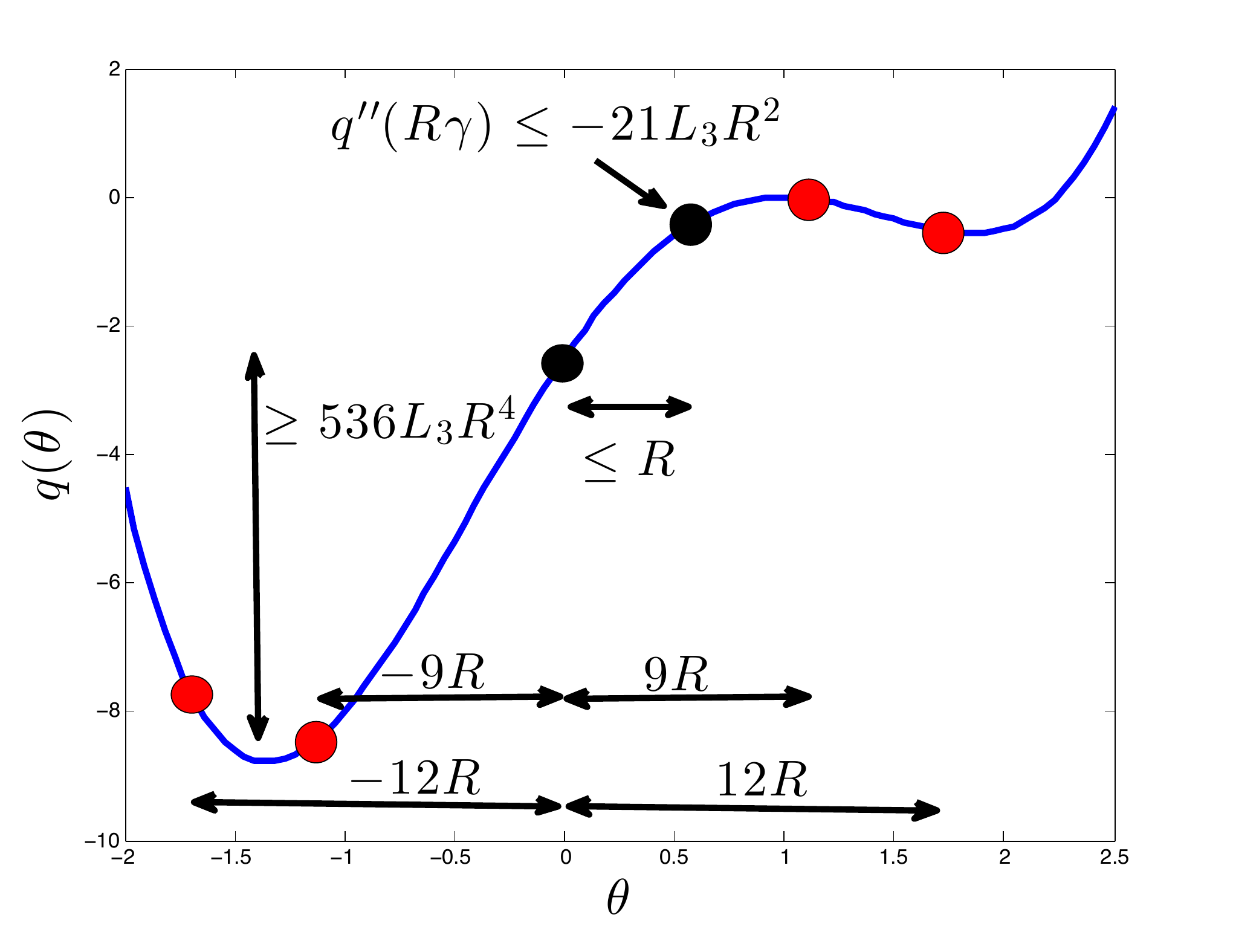}
\caption{Visual depiction of the guarantee of Lemma~\ref{lemNCprogress}.}\label{diag:NC-progress}
\end{figure}

Figure~\ref{diag:NC-progress} illustrates Lemma~\ref{lemNCprogress}. In particular, given the function $q$ and some point $\gamma \in [-1,1]$ with $q''(R \gamma) \le -21 \LipThird R^2$, by taking the minimum of the four different as given in red we can guarantee at least $536 \LipThird R^4$ reduction in the function value.  The main idea behind the proof is that since the function $q$ has Lipschitz third derivatives we can approximate the function using cubic interpolation. Then using the existence of nonconvexity and the asymmetry of a cubic function we deduce the result. Notice that if $f : \reals^{\Dim} \rightarrow \reals$ has $\LipThird$-Lipschitz continuous third derivatives and $\| s \| = 1$ then we can immediately use Lemma~\ref{lemNCprogress} to analyze Algorithm~\ref{alg:ExploitNC}.

\section{A cutting plane algorithm for nonconvex optimization}\label{sec:main-alg}

This section combines the Algorithms from Section~\ref{cutting-plane-methods}-\ref{sec:exploiting-nonconvexity} to obtain our improved complexity results. First, we present Algorithm~\ref{alg:CuttingTrustRegion} that \emph{roughly} solves a trust-region problem using cutting planes, i.e.,
$$
\argmin_{x \in \ball{R}{z}}{ f(x) }.
$$
Inside Algorithm~\ref{alg:CuttingTrustRegion} we add a proximal term to the function $f$, i.e. write $\fHat(x) := f(x) + \frac{\alpha}{2} \norm{z - x}^2$. This a typical strategy in nonconvex optimization theory \citep{agarwal2016finding,carmon2018accelerated,jin2017accelerated,royer2018complexity} and ensures that when we detect nonconvexity of $\fHat$ this corresponds to a large violation of nonconvexity on the function $f$ which we can use to exploit to reduce the function value using Lemma~\ref{lemNCprogress}. Following Algorithm~\ref{alg:CuttingTrustRegion}  we present two different algorithms (\eqref{eq:alg:GuardedCuttingPlaneAlgorithm} and \eqref{eq:alg:CuttingQuarticRegularization}). These algorithms repeatedly call Algorithm~\ref{alg:CuttingTrustRegion} until finding a stationary point. However, \eqref{eq:alg:GuardedCuttingPlaneAlgorithm} only evaluates first and second derivatives; \eqref{eq:alg:CuttingQuarticRegularization} evaluates the first, second and third derivatives (but makes less frequent evaluations). 

\begin{algorithm2e}[H]
\SetAlgoLined
\KwData{$f$, $z$, $\epsilon$, $\LipFirst$, $\LipThird$, $R$}
\KwResult{$z\ind{+}, K$}

$\alpha \gets 21 \LipThird R^2$;

$\fHat(x) := f(x) + \frac{\alpha}{2} \norm{z - x}^2$\;

$\epsHat \gets \epsilon/2$ \;
 $\LipFirstHat \gets \LipFirst + \alpha$ \tcc{The Lipschitz constant of $\grad \fHat$}
 
$ \N \gets \frac{d}{\tau} \log\left( \frac{8 \LipFirstHat R}{\epsHat} \right)$\;

$S\ind{\N}, x\ind{0}, \dots, x\ind{\N} \gets $\CuttingPlaneMethod{$\fHat, z, \Radius, \N$}\;

$u, v, K \gets $\NonconvexityCertificate{$\fHat, S\ind{\N}, x\ind{0}, \dots, x\ind{\N}, \LipFirstHat, \epsHat, \Radius$}\;

 \uIf{$\| \grad f(u) \| \le \epsilon$}{ \tcc{approximate first-order stationary point}
            $p \gets \argmin_{s : \| s \| = 1} s^T \grad^2 f(u) s$ using Singular Value Decomposition.

\uIf{$p^T \grad^2 f(u) p \ge - \alpha$}{ \tcc{approximate second-order stationary point}
\Return{$u, K$}
}
\uElse{
\Return{ \ExploitNC{$f, u, p, \Radius$}, $K$}
}
            }
\uElseIf{$v = \emptyset$}{
\Return{$u$, $K$} 
}
\uElse{
\Return{ \ExploitNC{$f, \frac{u + v}{2}, \frac{v - u}{\| u - v\|}, \Radius$}, K}
}    
\caption{CuttingTrustRegion}\label{alg:CuttingTrustRegion}
\end{algorithm2e}

\def\valueR{\LipThird^{-1/3} \epsilon^{1/3} / 3}
\def\valueRsmall{\LipThird^{-1/3} \epsilon^{1/3} /24}
\def\zInit{z}
\def\zPlus{z\ind{+}}

Lemma~\ref{lemMain} shows that during a call to Algorithm~\ref{alg:CuttingTrustRegion} we either find a (second-order) stationary point, as we wanted or we make a significant amount of progress in reducing the function value. 

\def\ubEpsLem{(0,R \LipFirst/2 ]}

\begin{restatable}{lem}{lemMain}\label{lemMain}
Consider Algorithm~\ref{alg:CuttingTrustRegion}. Suppose that Assumption~\ref{assume:cutting-plane} holds.
Let $z \in \reals^{\Dim}$, $R, \LipFirst, \LipThird \in (0,\infty)$. Assume that $f : \reals^{\Dim} \rightarrow \reals$ has $\LipFirst$-Lipschitz first derivatives and $\LipThird$-Lipschitz third derivatives on the set $\ball{12 \Radius}{z}$. 

If $\| \grad f(\zPlus) \| \ge \epsilon$ or $\lambda_{\min}(\grad^2 f(\zPlus)) \le -\alpha$ then
\begin{flalign}\label{eq:progress-result}
f(z\ind{+}) \le f(z) - \min \left\{ 10 \LipThird R^4, \frac{\epsilon^2}{168 R^2  \LipThird }  \right\}.
\end{flalign}
Furthermore, the runtime of Algorithm~\ref{alg:CuttingTrustRegion} is at most
$$
O\left( \frac{(\TimeCentre + \TimeGrad + K d) d}{\tau} \plusLog\left( \frac{R \LipFirst}{\epsilon}  \right) + \TimeHess + \Dim^{\omega} \right).
$$
\end{restatable}

The proof of Lemma~\ref{lemMain} is given in Section~\ref{sec:lemMain}. It is similar to Lemma~7 of \cite{carmon2017convex}.
Our algorithm simply consists of repeatedly calling Algorithm~\ref{alg:CuttingTrustRegion}, i.e.,
\begin{flalign}\label{eq:alg:GuardedCuttingPlaneAlgorithm}
z\ind{t+1}, K\ind{t} \gets \text{\CuttingTrustRegion{$f$, $z\ind{t}$,$\epsilon$, $\LipFirst$, $\LipThird$, $\valueR$}}.
\end{flalign}

\begin{thm}\label{thm:cutting}
Suppose that Assumption~\ref{assume:cutting-plane} holds. Let $z\ind{0} \in \reals^{\Dim}$, $\LipFirst, \LipThird \in (0,\infty)$. Assume that $f : \reals^{\Dim} \rightarrow \reals$ has $\LipFirst$-Lipschitz first derivatives and $\LipThird$-Lipschitz third derivatives. Let $f(z\ind{0}) - \inf_{x \in \reals^{\Dim}} f(x) \le \Delta$. Under these conditions, the procedure \eqref{eq:alg:GuardedCuttingPlaneAlgorithm} starting with $t = 0$ will find a point
\begin{flalign}\label{eq:termination-conditions}
\| \grad f(z\ind{m}) \| \le \epsilon \quad \quad \lambda_{\min}(\grad^2 f(z\ind{m})) \ge -\LipThird^{1/3} \epsilon^{2/3},
\end{flalign}
and uses computation time bounded above by
$$
O \left(  \left( \Delta \LipThird^{1/3} \epsilon^{-4/3} + 1 \right) \left( \frac{(\TimeCentre + \TimeGrad + \bar{K} d ) d}{\tau} \plusLog\left( \frac{\LipFirst^3}{\epsilon^2 \LipThird} \right) + \TimeHess + \Dim^{\omega} \right) \right),
$$
where $\bar{K}$ is a positive random variable satisfying $\Prob(\bar{K} \ge y) \le \e^{\frac{1-y}{10}}$ for all $y \in \reals$.

\end{thm}

\begin{proof}
Let $\bar{K} = \frac{1}{m} \sum_{t=1}^{m} K\ind{t}$ that $\Prob(\bar{K} \ge y) \le \e^{\frac{1-y}{10}}$ follows from standard Chernoff bound arguements (Lemma~\ref{lem:chernoff} Appendix~\ref{sec:chernoff}). Substituting $R = \valueR$ (chosen to maximize worst-case progress at each iteration) into Lemma~\ref{lemMain} shows we reduce the function by
$\LipThird^{-1/3} \epsilon^{4/3} / 20$ at each iteration that we do not terminate. Using $f(z\ind{0}) - \inf_{x \in \reals^{\Dim}} f(x) \le \Delta$ we deduce the number of iterations is at most $O(\Delta \LipThird^{1/3} \epsilon^{-4/3} + 1)$. The result follows by multiplying $O(\Delta \LipThird^{1/3} \epsilon^{-4/3} + 1)$ by the bound on the runtime of \CuttingTrustRegion given in Lemma~\ref{lemMain}.
 \end{proof}
 
 Using the John Ellipsoid \citep{johnEllipsoid} as the \Centre{ } function makes $\tau$ a dimension independent constant and $\TimeCentre = \tilde{O}(\Dim^{1 + \omega})$ (by casting it as an SDP \citep{boyd2004convex} then solving using \cite{lee2015faster}). This simplifies the expected runtime bound of Theorem~\ref{thm:cutting} to
\begin{flalign}\label{eq:our-runtime-bound}
\tilde{O} \left( \left( \Delta \LipThird^{1/3} \epsilon^{-4/3} + 1 \right) \left( (d^{\omega} + \TimeGrad) d  + \TimeHess \right)  \right)
\end{flalign}
as stated in Table~\ref{table-results} where $\tilde{O}$ omits log factors. 

For simplicity the proofs in this paper only apply to a restricted class of cutting plane methods: those that prove convergence by inducing a $\tau$-factor decrease in the volume at each iteration. However, some cutting plane methods such as volumetric centre use more sophisticated arguments. With technical but uninteresting modifications to our proofs we believe these more sophisticated cutting plane methods could be incorporated. For example, adapting the cutting plane method of \cite{lee2015faster} allows us to improve \eqref{eq:our-runtime-bound} by replacing the $d^{\omega}$ term with $d^2$.

From carefully reading the proof of Lemma~\ref{lemMain} one observes that replacing the code inside of the ``if $\| \grad f(u) \| \le \epsilon$'' statement of Algorithm~\ref{alg:CuttingTrustRegion} with ``\Return{$u, K$}'' causes the $\TimeHess$ term disappears from the runtime. The benefit of this replacement is that the algorithm becomes a first-order method, the downside is that we lose our second-order guarantees. Therefore, with this replacement, the bound \eqref{eq:our-runtime-bound} improves on the $\tilde{O} \left( (\Delta \LipFirst^{1/2} \LipThird^{1/6} \epsilon^{-5/3} + 1) \TimeGrad  \right)$ bound of \cite{carmon2017convex} when the dimension is small, the Lipschitz constant $\LipFirst$ is large, or high accuracy is desired. However, unlike \cite{carmon2017convex}
our results still require randomization, even when we only wish to obtain first-order guarantees. 

We also remark that our second-order guarantee given in \eqref{eq:termination-conditions} matches the second-order guarantee given by \cite{cartis2017improved} for quartic regularization. Our runtime for achieving second-order stationarity is a straightforward consequence of the efficient negative curvature exploitation proposed in \cite{carmon2017convex}.

Recall that the runtime of $p$th order regularization is
\begin{flalign}\label{pth-order-regularization-bound}
O\left(  (\TimeP + ?) \Delta \pLip^{1/p} \epsilon^{-\frac{p+1}{p}} \right)
\end{flalign}
where ? denotes the runtime of solving a $p$th order regularization problem.
Let us compare \eqref{eq:our-runtime-bound} and \eqref{pth-order-regularization-bound}.
Consider the case $p=2$, i.e., cubic regularization where ? can be replaced by $\Dim^{\omega}$. Note that if the Lipschitz constants and dimension are fixed and $\epsilon$ goes to zero then our runtime bound \eqref{eq:our-runtime-bound} is better than \eqref{pth-order-regularization-bound}. Furthermore, consider a problem with $\Dim^{\omega} \le \TimeGrad \approx \Dim \TimeHess$, i.e., the Hessian is computed via finite differences and the derivatives are expensive to evaluate. In this case, if the Lipschitz constants and the dimension grows, our algorithm has the same dimension dependence as cubic regularization, but an improved $\epsilon$ dependence. Next, consider \eqref{pth-order-regularization-bound} with $p=3$. As we stated in the introduction unlike $p$th order regularization our runtimes incorporate computational cost. However, there are even gains in terms of the evaluation complexity. In particular, suppose the high order derivatives are computed with finite differences of the gradients. In this case $\TimeCube = \Theta( \Dim^2 \TimeGrad )$. Hence quartic regularization requires $O\left(  \Dim^3 \Delta \LipThird^{1/3} \epsilon^{-\frac{4}{3}} \right)$ gradient evaluations versus $\tilde{O}\left(  \Dim^2 \Delta \LipThird^{1/3} \epsilon^{-\frac{4}{3}} \right)$ for our method --- a factor of $d$ improvement. 

It is difficult to provide a direct comparison between the runtime from Theorem~\ref{thm:cutting} and $p$th order regularization without making assumptions on the values of $\TimeGrad$, $\TimeHess$ and $\TimeCube$. Therefore to present a direct runtime comparison with \cite{birgin2017worst}, we use \eqref{eq:alg:CuttingQuarticRegularization} which avoids gradient calls by solving quartic regularization models. The ideas is just to run our algorithm on the quartic regularized subproblems as follows
\begin{subequations}\label{eq:alg:CuttingQuarticRegularization}
\begin{flalign}
\bar{f}\ind{t}(x) &:= \tilde{f}_{3}(z\ind{t}, x) \\
z\ind{t+1} &\gets \text{\CuttingTrustRegion{$\bar{f}\ind{t}$, $z\ind{t}$, $\frac{\epsilon}{2}$, $2 \LipFirst$, $2 \LipThird$, $\valueRsmall$}}.
\end{flalign}
\end{subequations}
Recall the definition of $\tilde{f}_{p}$ from \eqref{pth-order-regularization-function} in Section~\ref{sec:intro}. Theorem~\ref{thmQuarticReg} shows that by invoking \eqref{eq:alg:CuttingQuarticRegularization} we can obtain exactly the same evaluation complexity as quartic regularization while having a computationally runtime with the same $\epsilon$-dependence (up to log factors) as the evaluation complexity. 

To simplify the analysis and final runtime bounds in Theorem~\ref{thmQuarticReg} we assume that
\def\gdFasterRegime{\LipFirst^{3/2}/ \LipThird^{1/2}}
\begin{flalign}
\epsilon &\le \gdFasterRegime. \label{eq:gd-faster-regime}
\end{flalign}
This ensures that we ignore uninteresting corner cases in our analysis. In particular, if \eqref{eq:gd-faster-regime} is violated then $\Delta \LipThird^{1/3} \epsilon^{-4/3} \ge \Delta \LipFirst \epsilon^{-2}$. Hence the iteration bound of gradient descent will be better than the runtime bound of quartic regularization --- in which case one should run gradient descent.

\begin{restatable}{thm}{thmQuarticReg}\label{thmQuarticReg}
Suppose that Assumption~\ref{assume:cutting-plane} holds. Let $z\ind{0} \in \reals^{\Dim}$, $\LipFirst, \LipThird \in (0,\infty)$ and $\epsilon \in (0,\gdFasterRegime]$. Assume that $f : \reals^{\Dim} \rightarrow \reals$ has $\LipFirst$-Lipschitz first derivatives and $\LipThird$-Lipschitz third derivatives. Let $f(z\ind{0}) - \inf_{x \in \reals^{\Dim}} f(x) \le \Delta$. Under these conditions, the procedure \eqref{eq:alg:CuttingQuarticRegularization} starting with $t = 0$ finds a point $z\ind{m}$ such that \eqref{eq:termination-conditions} holds with computational time upper bounded by
$$
O \left( \left( \Delta \LipThird^{1/3} \epsilon^{-4/3} + 1 \right)\left( \TimeCube + \frac{(\TimeCentre + \Dim^3 + \bar{K} \Dim) \Dim}{\tau} \plusLog\left( \frac{\LipFirst^3}{\epsilon^2 \LipThird} \right)  \right) \right)
$$
where $\bar{K}$ is a positive random variable satisfying $\Prob(\bar{K} \ge y) \le \e^{\frac{1-y}{10}}$ for all $y \in \reals$.
\end{restatable}

The proof of Theorem~\ref{thmQuarticReg} appears in Appendix~\ref{proof:thmQuarticReg}. In Theorem~\ref{thmQuarticReg} we use the fact that in Lemma~\ref{lemMain} we only need the function $f$ to be Lipschitz on $\ball{12 R}{z\ind{t}}$. This allows us to get around the issue that the regularization term $\frac{2 \pLip}{(p+1)!} \| z\ind{t} - x \|^{p}$ has Lipschitz first derivatives on $\ball{12 R}{z\ind{t}}$ but not on $\reals^{\Dim}$.

Finally, we remark that both Theorem~\ref{thm:cutting} and \ref{thmQuarticReg} provide stochastic bounds on the runtime. However, the uncertainty in our runtime bound only occurs in the computational complexity since the random variable $\bar{K}$ is not multiplied by any $\TimeP$ term. Therefore the bound on the number of evaluations of the function and its derivatives is deterministic (given the algorithm terminates which occurs almost surely). Let us contrast the stochastic nature of our results with literature. The literature contains deterministic results. For example, \cite{birgin2017worst} has the same $\epsilon$-dependence as our work but only bounds the evaluation complexity, and \cite{carmon2017convex} has a worse $\epsilon$-dependence but better dependence on the problem dimension. Some literature also provides stochastic bounds. The work of \citep{agarwal2016finding,carmon2018accelerated} provides algorithms that, with high probability, find a second-order stationary point. The runtime is deterministic but there is a small probability that they fail to find a second-order stationary point. Our results can also be restated in a similar manner by the following simple modification to our algorithms. Fix some $\delta \in (0,1)$ and stop our algorithm when the computation time exceeds the upper bound proved in our theorems by a factor of $1 + 10 \log(1/\delta)$. With this modification our algorithms fail with probability at most $\delta$ and the runtime bounds hold almost surely. 

\section{Discussion}\label{sec:discussion}

Cutting plane methods for convex optimization have had practical success solving problems poorly conditioned problems of mild dimension. The classic example is the traveling salesperson problem. The linear program solved during branch and bound process can be solved with millions of variables \citep{padberg1991branch}. Another application of cutting plane methods in convex optimization is to two-stage stochastic programs. These problems are decomposed into a smaller but poorly conditioned master problem solved using a cutting plane method \citep{bahn1995cutting}. Large-scale nonconvex stochastic programs arise in optimal AC power flow \citep{knitroOptimalPower}.
For reasons similar to why cutting plane methods have been successful in convex optimization, this problem offers an opportunity for the application of cutting plane methods.

However, to develop a practical cutting plane method would require overcoming many hurdles not addressed in this theoretical paper. These hurdles include handling constraints and the fact that Lipschitz constants are unknown. To address this latter hurdle we believe one could use ideas from \cite{cartis2011adaptive} and \cite{birgin2017worst} which do not require knowledge of Lipschitz constants. It would also be amiable for the method to obtain the convex iteration bound of $O(\frac{d}{\tau} \log(R/r))$ on convex functions without need for the user to specify that the function was convex. This property would likely encourage fast local convergence.

\ifARXIVorSIAM
\bibliographystyle{abbrvnat}
\fi

\ifCOAA
\bibliographystyle{abbrvnat}
\fi

\bibliography{library-cutting-convex-until-guilty.bib}

\begin{thebibliography}{26}
\providecommand{\natexlab}[1]{#1}
\providecommand{\url}[1]{\texttt{#1}}
\expandafter\ifx\csname urlstyle\endcsname\relax
  \providecommand{\doi}[1]{doi: #1}\else
  \providecommand{\doi}{doi: \begingroup \urlstyle{rm}\Url}\fi

\bibitem[Agarwal et~al.(2017)Agarwal, Allen-Zhu, Bullins, Hazan, and
  Ma]{agarwal2016finding}
N.~Agarwal, Z.~Allen-Zhu, B.~Bullins, E.~Hazan, and T.~Ma.
\newblock Finding approximate local minima for nonconvex optimization in linear
  time.
\newblock \emph{Symposium on Theory of Computing}, 2017.

\bibitem[Atkinson and Vaidya(1995)]{atkinson1995cutting}
D.~S. Atkinson and P.~M. Vaidya.
\newblock A cutting plane algorithm for convex programming that uses analytic
  centers.
\newblock \emph{Mathematical Programming}, 69\penalty0 (1-3):\penalty0 1--43,
  1995.

\bibitem[Bahn et~al.(1995)Bahn, du~Merle, Goffin, and Vial]{bahn1995cutting}
O.~Bahn, O.~du~Merle, J.-L. Goffin, and J.-P. Vial.
\newblock A cutting plane method from analytic centers for stochastic
  programming.
\newblock \emph{Mathematical Programming}, 69\penalty0 (1-3):\penalty0 45--73,
  1995.

\bibitem[Birgin et~al.(2017)Birgin, Gardenghi, Mart{\'\i}nez, Santos, and
  Toint]{birgin2017worst}
E.~G. Birgin, J.~Gardenghi, J.~M. Mart{\'\i}nez, S.~A. Santos, and P.~L. Toint.
\newblock Worst-case evaluation complexity for unconstrained nonlinear
  optimization using high-order regularized models.
\newblock \emph{Mathematical Programming}, 163\penalty0 (1-2):\penalty0
  359--368, 2017.

\bibitem[Boyd and Vandenberghe(2004)]{boyd2004convex}
S.~Boyd and L.~Vandenberghe.
\newblock \emph{Convex optimization}.
\newblock Cambridge university press, 2004.

\bibitem[Carmon et~al.(2017{\natexlab{a}})Carmon, Duchi, Hinder, and
  Sidford]{carmon2017convex}
Y.~Carmon, J.~C. Duchi, O.~Hinder, and A.~Sidford.
\newblock `{C}onvex until proven guilty': {d}imension-free acceleration of
  gradient descent on non-convex functions.
\newblock In \emph{Proceedings of 34th International Conference on Machine
  Learning}, pages 654--663, 2017{\natexlab{a}}.

\bibitem[Carmon et~al.(2017{\natexlab{b}})Carmon, Duchi, Hinder, and
  Sidford]{carmon2017lowerII}
Y.~Carmon, J.~C. Duchi, O.~Hinder, and A.~Sidford.
\newblock Lower bounds for finding stationary points {II}: First-order methods.
\newblock \emph{arXiv preprint arXiv:1711.00841}, 2017{\natexlab{b}}.

\bibitem[Carmon et~al.(2018)Carmon, Duchi, Hinder, and
  Sidford]{carmon2018accelerated}
Y.~Carmon, J.~C. Duchi, O.~Hinder, and A.~Sidford.
\newblock Accelerated methods for nonconvex optimization.
\newblock \emph{SIAM Journal on Optimization}, 28\penalty0 (2):\penalty0
  1751--1772, 2018.

\bibitem[Carmon et~al.(2019)Carmon, Duchi, Hinder, and
  Sidford]{carmon2017lower}
Y.~Carmon, J.~C. Duchi, O.~Hinder, and A.~Sidford.
\newblock Lower bounds for finding stationary points {I}.
\newblock \emph{To appear in Mathematical Programming, arXiv preprint
  arXiv:1710.11606}, 2019.

\bibitem[Cartis et~al.(2011)Cartis, Gould, and Toint]{cartis2011adaptive}
C.~Cartis, N.~I. Gould, and P.~L. Toint.
\newblock Adaptive cubic regularisation methods for unconstrained optimization.
  {P}art {I}: motivation, convergence and numerical results.
\newblock \emph{Mathematical Programming}, 127\penalty0 (2):\penalty0 245--295,
  2011.

\bibitem[Cartis et~al.(2017)Cartis, Gould, and Toint]{cartis2017improved}
C.~Cartis, N.~I. Gould, and P.~L. Toint.
\newblock Improved second-order evaluation complexity for unconstrained
  nonlinear optimization using high-order regularized models.
\newblock \emph{arXiv preprint arXiv:1708.04044}, 2017.

\bibitem[Demmel et~al.(2007)Demmel, Dumitriu, and Holtz]{demmel2007fast}
J.~Demmel, I.~Dumitriu, and O.~Holtz.
\newblock Fast linear algebra is stable.
\newblock \emph{Numerische Mathematik}, 108\penalty0 (1):\penalty0 59--91,
  2007.

\bibitem[Goffin et~al.(1997)Goffin, Gondzio, Sarkissian, and
  Vial]{goffin1997solving}
J.-L. Goffin, J.~Gondzio, R.~Sarkissian, and J.-P. Vial.
\newblock Solving nonlinear multicommodity flow problems by the analytic center
  cutting plane method.
\newblock \emph{Mathematical programming}, 76\penalty0 (1):\penalty0 131--154,
  1997.

\bibitem[Jin et~al.(2017)Jin, Netrapalli, and Jordan]{jin2017accelerated}
C.~Jin, P.~Netrapalli, and M.~I. Jordan.
\newblock Accelerated gradient descent escapes saddle points faster than
  gradient descent.
\newblock \emph{arXiv preprint arXiv:1711.10456}, 2017.

\bibitem[John(1948)]{johnEllipsoid}
F.~John.
\newblock Extremum problems with inequalities as subsidiary conditions, 1948.

\bibitem[Lee et~al.(2015)Lee, Sidford, and Wong]{lee2015faster}
Y.~T. Lee, A.~Sidford, and S.~C.-w. Wong.
\newblock A faster cutting plane method and its implications for combinatorial
  and convex optimization.
\newblock In \emph{Foundations of Computer Science (FOCS), 2015 IEEE 56th
  Annual Symposium on}, pages 1049--1065. IEEE, 2015.

\bibitem[Levin(1965)]{CenterGravity}
A.~Y. Levin.
\newblock On an algorithm for the minimization of convex functions.
\newblock \emph{Soviet Math. Doklady}, 1965.

\bibitem[Nemirovski and Yudin(1983)]{NemirovskiYu83}
A.~Nemirovski and D.~Yudin.
\newblock \emph{Problem Complexity and Method Efficiency in Optimization}.
\newblock Wiley, 1983.

\bibitem[Nesterov(1983)]{nesterov1983method}
Y.~Nesterov.
\newblock A method of solving a convex programming problem with convergence
  rate ${O}(1/k^2)$.
\newblock \emph{Soviet Mathematics Doklady}, 27\penalty0 (2):\penalty0
  372--376, 1983.

\bibitem[Nesterov and Polyak(2006)]{nesterov2006cubic}
Y.~Nesterov and B.~T. Polyak.
\newblock Cubic regularization of {N}ewton method and its global performance.
\newblock \emph{Mathematical Programming}, 108\penalty0 (1):\penalty0 177--205,
  2006.

\bibitem[Padberg and Rinaldi(1991)]{padberg1991branch}
M.~Padberg and G.~Rinaldi.
\newblock A branch-and-cut algorithm for the resolution of large-scale
  symmetric traveling salesman problems.
\newblock \emph{SIAM review}, 33\penalty0 (1):\penalty0 60--100, 1991.

\bibitem[Plantenga(2006)]{knitroOptimalPower}
T.~Plantenga.
\newblock {KNITRO} for nonlinear optimal power flow applications, October 2006.
\newblock URL
  \url{https://www.artelys.com/downloads/pdf/composants-numeriques/knitro/papers/case_OPF.pdf}.

\bibitem[Press et~al.(1992)Press, Teukolsky, Vetterling, and
  Flannery]{press1992numerical}
W.~H. Press, S.~A. Teukolsky, W.~T. Vetterling, and B.~P. Flannery.
\newblock Numerical recipes in fortran 77, vol. 1.
\newblock \emph{New York, NY: Press Syndicate of the University of Cambridge},
  1992.

\bibitem[Royer and Wright(2018)]{royer2018complexity}
C.~W. Royer and S.~J. Wright.
\newblock Complexity analysis of second-order line-search algorithms for smooth
  nonconvex optimization.
\newblock \emph{SIAM Journal on Optimization}, 28\penalty0 (2):\penalty0
  1448--1477, 2018.

\bibitem[Shor(1977)]{shor1977cut}
N.~Z. Shor.
\newblock Cut-off method with space extension in convex programming problems.
\newblock \emph{Cybernetics}, 13\penalty0 (1):\penalty0 94--96, 1977.

\bibitem[Vaidya(1989)]{vaidya1989new}
P.~M. Vaidya.
\newblock A new algorithm for minimizing convex functions over convex sets.
\newblock In \emph{Foundations of Computer Science, 1989., 30th Annual
  Symposium on}, pages 338--343. IEEE, 1989.

\end{thebibliography}

\appendix

\section{Proof of Lemma~\ref{lemNCprogress}}\label{sec:lem:NC-progress}

To simplify the argument we first prove Lemma~\ref{lem:dim-free}, a dimension-free variant of Lemma~\ref{lemNCprogress}.
 
\begin{lem}\label{lem:dim-free}
Suppose the function $h : [-12,12] \rightarrow \reals$ has is $1$-Lipschitz continuous third derivatives and $h''(\gamma) \le -21$ for some $\gamma \in [-1,1]$ then
$$
\min\{  h(12), h(9),  h(-9), h(-12)  \} \le h(0) - 536.
$$ 
\end{lem}

\begin{proof}
Let $\tilde{h}(\theta) := h(0) +   h'(0) \theta  + \frac{h''(0)}{2} \theta^2  + \frac{h'''(0)}{6}  \theta^3$ by Lipschitz continuity we have
$$\abs{ \tilde{h}''(\theta) - h''(\theta) } \le \frac{\theta^2}{2} \quad \quad \abs{ \tilde{h}(\theta) - h(\theta) } \le \frac{\theta^4}{24}.$$ 
To ensure there exists $\gamma \in [-1,1]$ with $h''(\gamma) \le -21$ we need $\tilde{h}''(\gamma) \le -20$ and therefore $h'''(0) \theta + h''(0) = \tilde{h}''(\theta) \le  -20$ for $\theta = 1$ or $\theta = -1$. Consider the case $-h'''(0) + h''(0) \le  -20$, using this inequality and $\abs{ \tilde{h}(\theta) - h(\theta) } \le \frac{\theta^4}{24}$ we obtain
$$
h(\theta) - h(0) \le h'(0) \theta  + \frac{h''(0)}{2} \theta^2  + \frac{h'''(0)}{6}  \theta^3 + \frac{\theta^4}{24} \le  h'(0) \theta  + \theta^2 \frac{h'''(0) - 20 + \theta h'''(0)/3}{2} + \frac{\theta^4}{24},
$$
substituting in $\theta = -12$ and $\theta = 9$ we obtain the following two inequalities
\begin{flalign*}
h(-12) -  h(0)  &\le - 12 (h'(0)  + 18 h'''(0) ) - 12^2 \times 10 + 12^4/24 \\
h(9) - h(0) &\le 9 ( h'(0) + 18 h'''(0) ) - 9^2 \times 10 + 9^4/24.
\end{flalign*}
Hence 
$$\min\{h(-12) , h(9) \} \le h(0) -( h'(0) + 18 h'''(0) )  \min\{ -12, 9 \}  - 536 \le h(0)- 536.$$
Finally, the case $h'''(0) + h''(0) \le  -20$ follows by symmetry. In particular, defining $\hat{h}(\theta) := h(-\theta)$ and observing that
$-\hat{h}'''(0) = h'''(0)$ and $\hat{h}''(0) = h''(0)$ implies that the argument we just made (replacing $h$ with $\hat{h}$) shows
$$ \min\{h(12) , h(-9) \}  = \min\{\hat{h}(-12) , \hat{h}(9) \} \le h(0)- 536.$$
\end{proof}

\lemNCprogress*
 
\begin{proof}
Define $h(\theta) := \frac{1}{\LipThird R^4} q( \theta \Radius )$. Note that the function $h$ has 1-Lipschitz third Derivatives. Furthermore, since $q''(\gamma) \le -21 \LipThird R^2$ it follows that $h''(\gamma) = \frac{1}{\LipThird R^2}  q ''(  \Radius \gamma ) \le -21$. We conclude all the conditions of Lemma~\ref{lem:dim-free} are met.
\end{proof}

\section{Proof of Lemma~\ref{lemMain}}\label{sec:lemMain}

\lemMain*

\begin{proof}
Before beginning the proof we recap some useful facts:
\begin{subequations}\label{lemMain:basic-facts}
\begin{flalign} 
&\fHat(u) \le \fHat(\zInit) =  f(\zInit) \label{eq:monotone} \\
&f(u) - f(\zInit) = \fHat(u) - f(\zInit) - \frac{\alpha}{2} \| u  - \zInit \|^2 \le - \frac{\alpha}{2} \| u - \zInit \|^2, \label{eq:prox-inequality}
\end{flalign}
\end{subequations}
where \eqref{eq:monotone} is from Lemma~\ref{lem:NonconvexityCertificateWellDefined}, \eqref{eq:prox-inequality} follows from the definition of $\fHat$ and \eqref{eq:monotone}.

Consider the three possible outcomes of Algorithm~\ref{alg:CuttingTrustRegion} which are
\begin{subequations}
\begin{flalign}
& \| \grad f(u) \| \le \epsilon \quad  \text{and}  \quad  p^T \grad^2 f(u) p \ge - \alpha \label{lem:main:eq:terminate} \\
& \| \grad f(u) \| \le \epsilon \quad  \text{and} \quad  p^T \grad^2 f(u) p < - \alpha \label{lem:main:eq:second-order-fails} \\
& \| \grad f(u) \| > \epsilon.  \label{lem:main:eq:first-order-fails}
\end{flalign}
\end{subequations}
If \eqref{lem:main:eq:terminate} holds then Lemma~\ref{lemMain} clearly holds.
If \eqref{lem:main:eq:second-order-fails} or \eqref{lem:main:eq:first-order-fails} holds then we wish to establish \eqref{eq:progress-result}.

Let us show \eqref{eq:progress-result} when \eqref{lem:main:eq:second-order-fails} holds. In this case, $\zPlus \gets \ExploitNC{f, u, p, \Radius}$ and by Lemma~\ref{lemNCprogress},
$$
f(\zPlus) - f(\zInit) \le -536 \LipThird R^4.
$$

Let us show \eqref{eq:progress-result} when \eqref{lem:main:eq:first-order-fails} holds. Consider the three cases arising from Lemma~\ref{lem:NonconvexityCertificateWellDefined}. 

\begin{enumerate}[(i)]
\item $v = \emptyset$ and $\| \grad \fHat(u) \| \le \epsHat = \epsilon / 2$. In this case $u = \zPlus$. Therefore
$$
\epsilon \le \| \grad f(\zPlus) \| \le \| \grad \fHat(\zPlus) \| + \alpha \| \zPlus - \zInit \| \le\epsilon/2 + \alpha \| \zPlus - \zInit \|.
$$
Rearranging yields $\| \zPlus - \zInit \| \ge \epsilon / (2 \alpha)$. Therefore, using \eqref{eq:prox-inequality}, $\| \zPlus - \zInit \| \ge \epsilon / (2 \alpha)$, and $\alpha = 21 \LipThird R^2$ we get
$$
f(\zPlus) - f(\zInit) \le -\frac{\alpha}{2} \| \zPlus - \zInit \|^2 \le \frac{\epsilon^2}{8 \alpha} \le \frac{\epsilon^2}{168 \LipThird R^2}.
$$
\item $v = \emptyset$ and $\| \grad \fHat(u) \| > \epsHat = \epsilon / 2$.  In this case $u = \zPlus$.
By Lemma~\ref{lem:NonconvexityCertificateWellDefined} we have $\| u- \zInit \| > R$.
Therefore using \eqref{eq:prox-inequality}, $\| u- \zInit \| > R$, and $\alpha = 21 \LipThird R^2$ we get
$$f(\zPlus) - f(\zInit) \le -\frac{\alpha}{2} \| \zPlus - \zInit \|^2 = -\frac{\alpha R^2}{2} = -\frac{21}{2} \LipThird R^4.$$
\item $v \neq \emptyset$. In this case, we have a certificate of nonconvexity:
\begin{flalign*}
\fHat(u) < \fHat(v) + \grad \fHat(v)^T (v - u) \Rightarrow f(u) < f(v) + \grad f(v)^T (v - u) - \frac{\alpha}{2} \| v - u \|^2.
\end{flalign*}
Let $q(\theta) :=  f \left(c + \theta s \right)$ with $s= \frac{v - u}{\| u - v\|}$ and $c = \frac{u + v}{2}$. We deduce there exists some point $\gamma \in [-1,1]$ with $q''(\gamma) < -\alpha$. Since $\alpha = 21 \LipThird R^2$ in Algorithm~\ref{alg:CuttingTrustRegion} we can apply Lemma~\ref{lemNCprogress} to show that we reduce the function by at least $536 \LipThird R^4$ during our call to $\ExploitNC{$f, c, s, \Radius$}$.
\end{enumerate}
Therefore if \eqref{lem:main:eq:second-order-fails} or \eqref{lem:main:eq:first-order-fails} holds then \eqref{eq:progress-result} holds.

It remains to derive the runtime of the algorithm per iteration. We can bound the computational cost by
$$
O\left(  K \N  d + \N  (\TimeCentre + \TimeGrad) + \TimeHess + \Dim^{\omega} \right)
$$
where $K$ is the random variable arising from \eqref{eq:call-cut-then-cert}. The term $O( \N K d)$ come from the fact that it requires $O(\N d)$ to evaluate if $u \not\in S\ind{\N}$. The term $\TimeCentre + \TimeGrad$ represents the cost of each iteration of Algorithm~\ref{alg:CuttingPlaneMethod}. The term $O(\TimeHess + \Dim^{\omega})$ comes from the fact that at each iteration of Algorithm~\ref{alg:CuttingTrustRegion} we compute an SVD which takes $O(\Dim^{\omega})$ time \citep[Section~2.6]{press1992numerical} and evaluate the Hessian. Using Lemma~\ref{lem:NonconvexityCertificateWellDefined} with $\N = \Big\lfloor \frac{d}{\tau} \plusLog\left( \frac{8 \LipFirstHat R}{\epsHat} \right) \Big\rfloor$ we know the computation cost can be bounded by
$$
O\left( \frac{d  (K d  + \TimeCentre + \TimeGrad)}{\tau} \plusLog\left( \frac{8 \LipFirstHat R}{\epsHat} \right)  + \TimeHess + \Dim^{\omega} \right).
$$
\end{proof}

\section{Proof of Lemma~\ref{lem:chernoff}}\label{sec:chernoff}
\begin{lem}\label{lem:chernoff}
Let $\bar{K} = \frac{1}{m} \sum_{t=1}^{m} K\ind{t}$ with independent random variables $K\ind{t} \sim \Geo(p\ind{t})$ and $p\ind{t} \ge 1/2$ then $\Prob(\bar{K} \ge y) \le \e^{\frac{1-y}{10}}$ for all $y \in \reals$.

\end{lem}
\begin{proof}
Since $K\ind{t} \sim \Geo(p\ind{t})$ with $p\ind{t} \ge 1/2$ we can bound the moment generating function for $\alpha \le 1/10$:
$\Expect[\e^{K\ind{t} \alpha}] = \frac{1}{1 - \frac{1 - \e^{-\alpha}}{p\ind{t}} } \le \frac{2}{1 + \e^{-\alpha} } \le 1 + 5 \alpha$.
Using a typical Chernoff bound arguement,
\begin{flalign*}
\Prob(\bar{K} \ge y) &=  \Prob(\e^{\bar{K}} \ge \e^y)  \le \frac{\Expect[e^{\bar{K}/10}]}{e^{y/10}} = \frac{\Pi_{t=1}^{m} \Expect[e^{K\ind{t}/(10m)}]}{e^{y/10}} \\
& \le \frac{\left( 1 + 1/(5m) \right)^{m}}{e^{y/10}} \le e^{\frac{1-y}{10}}.
\end{flalign*}

\end{proof}

\section{Proof of Theorem~\ref{thmQuarticReg}}\label{proof:thmQuarticReg}

\thmQuarticReg*

\begin{proof} 
Define $\bar{K} := \frac{1}{m} \sum_{t=1}^{m} K\ind{t}$, $\Prob(\bar{K} \ge y) \le \e^{\frac{1-y}{10}}$ follows from Lemma~\ref{lem:chernoff}.
 
Let us check the assumptions of Lemma~\ref{lemMain} hold. Recall that we defined $\bar{f}$ such that
$$
\bar{f}\ind{t}(x) = f(x\ind{t}) +  \grad f(x\ind{t})^T (x - x\ind{t}) + \dots + \frac{\LipThird}{12} \| x - x\ind{t} \|^{4}.
$$
Therefore $\bar{f}\ind{t}(x)$ has $2 \LipThird$-Lipschitz third derivatives.

For $x \in \ball{12 R}{z\ind{t}}$ with $R = \valueRsmall$ we have 
$$
\| \grad f(x) - \grad \bar{f}\ind{t}(x) \| \le \frac{\LipThird}{3} \| x - z\ind{t} \|^{3} \le \frac{\epsilon}{24}.
$$
Therefore if $\| \grad \bar{f}\ind{t}(z\ind{t+1}) \| \le  \epsilon/2$ then $\| \grad f(z\ind{t+1}) \| \le \epsilon$. Similarly, for any $x, x' \in \ball{12 R}{z\ind{t}}$ we have
$$
\| \grad^2 f(x) - \grad^2 \bar{f}\ind{t}(x) \| \le \LipThird \| x - z\ind{t} \|^{2} \le \frac{\LipThird^{1/3} \epsilon^{2/3}}{4}
$$
it follows that $\lambda_{\min}(\grad^2 f(z\ind{t+1})) \ge -\LipThird^{1/3} \epsilon^{2/3}/2$ if $\lambda_{\min}(\grad^2 \bar{f}\ind{t}(z\ind{t+1})) \ge -\LipThird^{1/3} \epsilon^{2/3}$. Furthermore, we deduce that $\bar{f}\ind{t}$ has $2 \LipFirst$-Lipschitz first derivatives using $\epsilon \le  \LipFirst^{3/2}/\LipThird^{1/2} \Rightarrow \LipThird^{1/3} \epsilon^{2/3} \le \LipFirst$.

With these conditions established we can apply Lemma~\ref{lemMain} with $R = \valueRsmall$ to deduce $\bar{f}\ind{t}(z\ind{t}) - \bar{f}\ind{t}(z\ind{t+1})  =\Omega(  \LipThird^{-1/3} \epsilon^{4/3}  )$. This translates into a progress bound on $f$ since
$$
\bar{f}\ind{t}(z\ind{t}) - \bar{f}\ind{t}(z\ind{t+1}) = f(z\ind{t}) -  \bar{f}\ind{t}(z\ind{t+1}) \le f(z\ind{t}) -  f(z\ind{t+1}).
$$
Therefore if $m$ is the total number of iterations we have
$$
\Delta \ge f(z\ind{0}) - f(z\ind{m}) = \sum_{t=0}^{m-1}(f(z\ind{t}) -  f(z\ind{t+1})) = \Omega( m  \LipThird^{-1/3} \epsilon^{4/3}  ).
$$
Rearranging shows $m = O( \Delta \LipThird^{1/3} \epsilon^{-4/3} + 1)$. The computational cost per iteration derives from Lemma~\ref{lemMain} using $\TimeGrad = O(\Dim^3)$, since we need to evaluate the gradient of a quartic regularized model at each iteration.
\end{proof}

\end{document}